\documentclass[11pt]{amsart}
\usepackage{amssymb, amsmath,amsmath,latexsym,amssymb,amsfonts,amsbsy, amsthm,mathtools,graphicx,CJKutf8,CJKnumb,CJKulem,color,mathrsfs}

\numberwithin{equation}{section}
\newtheorem{thm}{Theorem}[section]
\newtheorem{lem}[thm]{Lemma}
\newtheorem{prop}[thm]{Proposition}

\theoremstyle{definition}
\newtheorem{defn}[thm]{Definition}
\newtheorem{rem}[thm]{Remark}

\allowdisplaybreaks

\begin{document}

\title[stochastic incompressible Navier-Stokes equations]{Local and global  strong solutions to the stochastic incompressible Navier-Stokes equations in critical Besov space}

\author{Lihuai Du}
\email{dulihuai@zju.edu.cn}
\address{School of Mathematical Sciences, Zhejiang University, Hangzhou 310027, China}

\author{Ting Zhang}
\email{zhangting79@zju.edu.cn}
\address{School of Mathematical Sciences, Zhejiang University, Hangzhou 310027, China}

\maketitle
\begin{abstract}
Considering the stochastic Navier-Stokes system in $\mathbb{R}^d$ forced by a multiplicative white noise,  we establish the local existence and uniqueness of the strong solution when the initial data take values in the critical space $\dot{B}_{p,r}^{\frac{d}{p}-1}(\mathbb{R}^d)$.
 The proof is based on the contraction mapping principle, stopping time and stochastic estimates. Then we prove the global existence of strong solutions in probability  if the initial data are sufficiently small, which contain  a class of highly oscillating ``large" data.
 \end{abstract}

\section{Introduction}
\label{sec:1}
We study the Navier-Stokes equations in the whole space $\mathbb{R}^d(d\geq2)$ forced by a multiplicative white noise

\begin{equation}
\label{1.1}
\left\{\begin{array}{l}
du+(u\cdot\nabla u-\Delta u+\nabla p)dt=f(t,u) dW_H,~~t>0,\\
\nabla\cdot u=0,\\
u|_{t=0}=u_0(x),
\end{array}
\right.
\end{equation}
where $u=(u_1,\ldots, u_d)$ and $p$ represent the velocity field and the pressure respectively. The system (\ref{1.1}) describes the flow of a viscous incompressible fluid. The addition of the white noise driven terms to the basic governing equations is nature
for both practical and theoretical applications. Such stochastically forced terms are used to account for numerical and empirical uncertainties and have been proposed as a model for turbulence. In order to emphasize the stochastic effects and for simplicity of exposition, we do not include a deterministic forcing in $(\ref{1.1})$, but note that all the results of this paper may be easily modified to include this more general case.

As a consequence of stochastic partial differential equations (SPDE) such as the stochastic Navier-Stokes equations are gaining more and more interest in fluid mechanical research. First result can be traced back to the pioneering work by of Bensoussan and Temam \cite{bensoussan1973equations} in  1973, after that, there have been a lot of studies on the existence and uniqueness of the strong solution for the
stochastic Navier-Stokes system, refer to    \cite{breckner1900galerkin,menaldi2002stochastic,da2014stochastic} and the references therein. Capinski and Peszat \cite{capinski1997local} obtained the local existence and uniqueness of the strong solution in three-dimensional bounded domain with sufficiently regular initial data. Taniguchi \cite{taniguchi2011existence} considered the existence and uniqueness of the strong solution to the two dimensional stochastic Navier-Stokes equations with non-Lipschitz forces in the unbounded domain by local monotonicity method,
Holz and Ziane \cite{glatt2009strong} obtained the local existence and uniqueness of the strong solution for the stochastic Navier-Stokes equations in two-dimensional or three-dimensional bounded domains forced by a multiplicative noise when the initial data are in $H^1$, and they proved the global existence in the two-dimensional case.  Kim \cite{kim2010strong} obtained the local and global existence of the strong solution for the three-dimensional stochastic Navier-Stokes equations when the initial data are in $H^{\alpha+\frac{1}{2}}(\mathbb{R}^3),0<\alpha<\frac{1}{2}$. On the other hand, many author consider the
martingale solutions to the stochastic Navier-Stokes equations,
see e.g.   \cite{capinski1994stochastic,capinski2001existence,flandoli1995martingale,flandoli1995ergodicity,
mikulevicius2004stochastic,mikulevicius2005global}.

If $f\equiv 0$, the system (\ref{1.1}) is the classical Navier-Stokes system
and is invariant with the scaling
 $$u_\lambda(x,t)=\lambda u(\lambda x,\lambda^2t),~~~~P_\lambda(x,t)=\lambda P(\lambda^2x,\lambda^2t),$$
 for any $\lambda>0$. Hence, we introduce the scaling $\varphi_\lambda(x)=\lambda\varphi(\lambda x).$ A function space $X$ on $\mathbb{R}^d$ is called critical to the Navier-Stokes equations if it satisfies the property $\|\varphi_\lambda\|_{{X}}=\|\varphi\|_{{X}}$. Examples of critical spaces for the Navier-Stokes equations are
    $$
    \dot{H}^{\frac{d}{2}-1}(\mathbb{R}^d),\ L^d(\mathbb{R}^d),\ \dot{B}_{p,r}^{\frac{d}{p}-1}(\mathbb{R}^d)\ (1\leq p,r\leq \infty),\ BMO^{-1}(\mathbb{R}^d),
    $$
    see   \cite{Chemin1999Th,Cannone93,fujita1964navier,kato1962nonstationary,koch2001well}.
Kato and Fujita \cite{fujita1964navier,kato1962nonstationary} established the existence and uniqueness of the mild solution when $u_0\in H^{\frac{d}{2}-1}$ and $L^d$, and J. -Y.
Chemin \cite{Chemin1999Th} considered in the Besov space $\dot{B}_{p,r}^{\frac{d}{p}-1}$ with $1\leq p<\infty, 1\leq r\leq \infty$, which contains a class of
highly oscillating \lq\lq  large\rq\rq initial data, for example,
    \begin{equation}\label{ep}
    u_0^\epsilon(x)=\varepsilon^{\frac{3}{p}-1}\sin(\frac{x_1}{\varepsilon})(0,-\partial_{x_3}\phi(x),\partial_{x_2}\phi(x)), \ \phi(x)\in \mathcal{S}(\mathbb{R}^3),\ p\in[3,\infty).
   \end{equation}
Therefore, there is a gap between the stochastic Navier-Stokes equations and the classical Navier-Stokes equations.

Our aim here is to extend the partial result of   \cite{Chemin1999Th} to stochastic case. We establish two main results for the system (\ref{1.1}). The first result  (Theorem \ref{C2}) addresses the local existence and uniqueness of the strong solution, which evolve continuously in the
Besov space $\dot{B}_{p,r}^{\frac{d}{p}-1}(\mathbb{R}^d)$, for any
\begin{equation}
2\leq p<\infty,
\
\left\{
\begin{array}{l}
2\leq r<\infty,~~~~~~~~~\text{if}~~2\leq p\leq d,\\
2\leq r<\frac{2p}{p-d},~~~~~~~\text{if}~~d<p<\infty.
\end{array}
\right.\label{pr}
\end{equation}

In the second result  (Theorem \ref{C3}), we turn to the issue of the global existence of the strong solutions with multiplicative noise in $d\geq 2$. When the noises have the special structure, then
for any $0<\epsilon<1$, there exists a $\delta(\epsilon)>0$, such that when $$\mathbb{E}\|u_0\|_{\dot{B}_{p,r}^{\frac{d}{p}-1}}^r\leq \delta(\epsilon),$$
we have
$$\mathbb{P}(u ~\text{is global})\geq 1-\epsilon.$$
Thus, there is a large probability that the system (\ref{1.1}) has the unique global solution with
the high oscillating initial velocity $u_0^\epsilon$ in (\ref{ep})   when $\epsilon$ is sufficient small.

The manuscript is organized as follows. In Section \ref{sec:2} we review some mathematical backgrounds, deterministic and stochastic, needed throughout the rest of
the work. Section \ref{sec:3} contains the precise definitions of solutions to $(\ref{1.1})$, along with statements of our main results. And we shall prove the global existence  of the strong solutions of the stochastic
modified equations (\ref{4.3}) by the contraction mapping principle in Section \ref{sec:4}. The final section is devote to proofs of the local existence and global existence
of  (\ref{1.1}). Appendix gather various additional technical tools used throughout the paper.
\section{Preliminaries}
\label{sec:2}
Here we recall some deterministic and stochastic ingredients which be used throughout this paper.
\subsection{Deterministic background}
\label{sec:2.1}
We first introduce Littlewood-Paley decomposition and  the definition of Besov spaces, which could refer to the nice book   \cite{bahouri2011fourier}.
Choose nonnegative smooth functions $\varphi$, supported respectively a ring $\mathcal{C}=\{\xi\in\mathbb{R}^d,\frac{3}{4}\leq\xi\leq\frac{8}{3}\}$ such that
 $$
 \sum\limits_{j\in\mathbb{Z}}\varphi(2^{-j}\xi)=1,~~~\xi\in\mathbb{R}^d\backslash\{0\},
 $$
 $$
|i-j|\geq 2\Rightarrow ~\text{Supp}~\varphi(2^{-i}\cdot)\cap~\text{Supp}~\varphi(2^{-j}\cdot)=\emptyset.
$$
The Littlewood-Paley blocks are defined as
$$\Delta_{j}u=\mathcal{F}^{-1}(\varphi(2^{-j})\mathcal{F}u),~~\forall j\in\mathbb{Z}.$$
Then $\Delta_j u=K_j\ast u$, where $K_j=\mathcal{F}^{-1}\varphi(2^{-j}\cdot)$. We also use the notation
$$S_j f=\sum_{i\leq j-1}\Delta_j f.$$
Let $\mathcal{S}'_h(\mathbb{R}^d)$ be the space of
tempered distributions $u$ such that
$$\lim\limits_{j\rightarrow{-\infty}}S_j u=0~~\text{in}~\mathcal{S}'.$$
We recall  the definitions of
homogeneous Besov space $\dot{B}_{p,r}^s(\mathbb{R}^d)$ and Chemin-Lerner-type spaces $\mathcal{L}_T^q(\dot{B}_{p,r}^s(\mathbb{R}^d))$.
\begin{defn}
Let $(p,r)\in [1,+\infty]^2$, $s\in \mathbb{R}$, we define
$$\dot{B}_{p,r}^s(\mathbb{R}^d)
:=\{u\in \mathcal{S}_h'(\mathbb{R}^d)|\|u\|_{\dot{B}_{p,r}^s}:=\{2^{js}\|\Delta_ju\|_{L^p}\}_{\ell^r}<\infty\}.$$
\end{defn}
\begin{defn}For $T>0$, $s\in \mathbb{R}$, and $1\leq q,r\leq \infty$, we set
$$\|u\|_{\mathcal{L}_T^q(\dot{B}_{p,r}^s)}:=\|2^{js}\|\Delta_ju\|_{L_T^q(L^p)}\|_{\ell^r(\mathbb{Z})}.$$
We then define the space $\mathcal{L}_T^q(\dot{B}_{p,r}^s)$ as the set of tempered distribution $u$ over $(0,T)\times\mathbb{R}^d$ such
that $\lim\limits_{j\rightarrow -\infty} S_j u=0$ in $L^q([0,T];L^\infty(\mathbb{R}^d))$ and
$\|u\|_{\mathcal{L}_T^q(\dot{B}_{p,r}^s)}<\infty$.
\end{defn}

\subsection{Background on stochastic analysis}
\label{sec;2.2}
We next briefly recall some aspects on the theory of the infinite dimensional stochastic analysis which we use below.
We refer the reader to   \cite{van2012maximal,van2012stochastic} for an extended treatment of the subject. For this purpose we start by fixing a stochastic
basis $(\Omega,\mathscr{F},\mathbb{P})$ endowed with filtration $\{\mathscr{F}_t\}_{t\geq 0}$. An
$\mathscr{F}$-cylindrical Brownian motion on a Hilbert space $H$  is a bounded linear operator $\mathcal{W}_H:
L^2(\mathbb{R}_{+};H)\rightarrow L^2(\Omega)$ such that

$(\expandafter{\romannumeral1})$ for all $t\geq 0$ and $h\in H$, the random variable $W_H(t)h:=\mathcal{W}_H(1_{(0,t]}\otimes h)$
is centred Gaussian and $\mathscr{F}_t$-measurable;

$(\expandafter{\romannumeral2})$ for all $t_1,t_2\geq 0$ and $h_1,h_2\in H$, we have $\mathbb{E}(W_H(t_1)h_1\cdot W_H(t_2)h_2)
=t_1\wedge t_2 [h_1,h_2]$;

$(\expandafter{\romannumeral3})$ for all $t_2\geq t_1\geq 0$ and $h\in H$, the random variable $W_H(t_2)h-W_H(t_1)h$ is independent of $\mathscr{F}_{t_1}$.

It is easy to see that for all $h\in H$, the process $(t,\omega)\rightarrow (W_H(t)h)(\omega)$
is an $\mathscr{F}$-Brownian motion (which is standard if $\|h\|=1$). Moreover, two such Brownian motion $(W_H(t)h_1)_{t\geq 0}$ and
$(W_H(t)h_2)_{t\geq 0}$ are independent if and only if $h_1$ and $h_2$ are orthogonal in $H$.

For $0\leq a<b<\infty$, $\mathscr{F}_a$-measurable sets $F\subset\Omega$, $h\in H$,and $f\in L^p(\mathbb{R}^d)$ the stochastic integral of the
indicator process $(t,\omega)\mapsto 1_{(a,b]\times F}f\otimes h$ with respect to $W_H$ is defined as the $L^p(\mathbb{R}^d)$-valued random variable
$$\int_0^t1_{(a,b]\times F}(f\otimes h)dW_H:=(W_H(t\wedge b)h-W_H(t\wedge a)h)1_Ff,~~~~t\geq 0.$$
By linearity, this definition extends to adapted finite rank step processes $G:\mathbb{R}_+\times \Omega\rightarrow L^p(\mathbb{R}^d;H)$,
which we define as finite linear combinations of adapted indicator process of the above form. Recall that a process
$G:\mathbb{R}_+\times\Omega\rightarrow L^p(\mathbb{R}^d;H)$ is called $\mathscr{F}$-adapted if for every $t\in\mathbb{R}_+,\omega\mapsto G(t,\omega)$
is $\mathscr{F}_t$-measurable.

The next result is a special case of Theorem 6.2 in   \cite{van2007conditions}.
\begin{prop}
\label{BGD}
Let $p\in (1,\infty)$ and $q\in (1,\infty)$ be fixed. For all $\mathscr{F}$-adapted finite rank step process $G:\mathbb{R}_+\times\Omega\rightarrow
L^q(\mathbb{R}^d,H)$, we have the \lq\lq It\^{o} isomorphism\rq\rq
\begin{equation}
\label{2.1}
 c^p\mathbb{E}\|G\|_{L^q(\mathbb{R}^d;L^2(\mathbb{R}_{+};H))}^p\leq\mathbb{E}\left\|\int_0^\infty GdW_H\right\|_{L^q(\mathbb{R}^d)}^p \leq C^p\mathbb{E}\|G\|_{L^q(\mathbb{R}^d;L^2(\mathbb{R}_{+};H))}^p
\end{equation}
with constants $0<c\leq C<\infty$ independent of $G$.
\end{prop}

By a standard density argument, these inequalities can be used to extend the stochastic integral to the Banach space
$L_{\mathscr{F}}^p(\Omega;L^q(\mathbb{R}^d;L^2(\mathbb{R}_+;H)))$ of all $\mathscr{F}$-adapted process $G:\mathbb{R}_+\times\Omega\rightarrow L^q(\mathbb{R}^d;H)$ which belong to $L^p(\Omega;L^q(\mathbb{R}^d;L^2(\mathbb{R}_+;H)))$. In the remainder of this paper, all stochastic integrals
are understood in the above sense.

By Minkoski's inequality,  for $q\in [2,\infty)$ one has
$$\|G\|_{L^q(\mathbb{R}^d;L^2(\mathbb{R}_+;H))}\leq \|G\|_{L^2(\mathbb{R}_+;L^q(\mathbb{R}^d;H))}.$$
In combination with Proposition \ref{BGD}, for $p\in(1,\infty)$ and $q\in[2,\infty)$, this gives the one-sided inequality
\begin{equation}
\label{2.2}
\mathbb{E}\bigg\|\int_0^\infty GdW_H\bigg\|_{L^q(\mathbb{R}^d)}^p \leq \mathbb{E}\|G\|^p_{L^2(\mathbb{R}_+;L^q(\mathbb{R}^d;H))}.
\end{equation}
Since we consider strong solutions of  the system (\ref{1.1}) evolving in $\dot{B}_{p,r}^{\frac{d}{p}-1}$, then we define
\begin{equation}
\label{b3}
\mathbb{B}_{p,r}^s:=\{G\in\mathcal{S}_h'(\mathbb{R}^d;H)|\|G\|^r_{\dot{\mathbb{B}}_{p,r}^s}:=\sum\limits_{j=-\infty}^{+\infty}2^{jrs}\|\Delta_j G\|^r_{L^p(\mathbb{R}^d;H)}<\infty\},
\end{equation}
where $1\leq p,r\leq \infty$, $s\in \mathbb{R}$.

\section{Main results}
\label{sec:3}
Recall that $\nabla p$ can be eliminated by projecting  onto the space of divergence free vector fields, using the
Leray projector
$$\mathbf{P}=\text{Id}+\nabla(-\Delta)^{-1}\text{div}.$$
By the divergence free condition, (\ref{1.1}) can be rewritten as
\begin{equation}
\label{3.1}
\left\{\begin{array}{l}
du+(\mathbf{P}\text{div}(u\otimes u)-\Delta u)dt=\mathbf{P}f(t,u)dW_H,\\
u|_{t=0}=u_0.
\end{array}
\right.
\end{equation}
With the mathematical preliminaries in hand, we make precise the definition of  local strong solutions of the stochastic Navier-Stokes equations.
\begin{defn}
\label{C1}
Suppose that $2\leq p,r\leq\infty$. Fix a stochastic basis $\mathcal{S}:=(\Omega,\mathscr{F},\mathbb{P},\{\mathscr{F}_t\}_{t\geq 0},W_H)$ and
$u_0$ an $\dot{B}_{p,r}^{\frac{d}{p}-1}$ valued $\mathscr{F}_0$-measurable random variable.

$(1)$ $(u,\tau)$ is called a local strong solution of the stochastic Navier-Stokes equations (\ref{1.1}) if the following conditions are satisfied:
\begin{itemize}
\item $u$ is right continuous progressively measurable process and $u\in L^r(\Omega;\mathcal{L}_{loc}^\infty((0,\infty);\dot{B}_{p,r}^{\frac{d}{p}-1}))\cap L^r(\Omega;L^r_{loc}((0,\infty)
;\dot{B}_{p,r}^{\frac{d}{p}-1+\frac{2}{r}}))$ ;

\item
\begin{equation}
\label{3.2}
\begin{split}
{\sigma}(\omega)&=\left\{\begin{array}{l}
\text{inf}\bigg\{0\leq t<\infty:\|u\|_{L_t^r\dot{B}_{p,r}^{\frac{d}{p}-1+\frac{2}{r}}}\geq R\bigg\},\\
\infty,\text{if the above set $\{\cdot\}$ is empty};
\end{array}
\right.\\
{\varrho}_N(\omega)&=\left\{\begin{array}{l}
\text{inf}\bigg\{0\leq t<\infty:\|u\|_{\dot{B}_{p,r}^{\frac{d}{p}-1}}\geq N\bigg\},\\
\infty,\text{if the above set $\{\cdot\}$ is empty};
\end{array}
\right.\\
\tau_N&=\min\{\sigma,\varrho_N\}
\end{split}
\end{equation}
where $R$ is  determined by (\ref{4.14}), and
\begin{equation}
\label{3.3}
\tau(\omega)=\lim\limits_{N\rightarrow\infty}\tau_N(\omega)~~\text{for almost all}~\omega.
\end{equation}
\item $u(t,x)\in C([0,\tau(\omega));\dot{B}_{p,r}^{\frac{d}{p}-1})$ for almost all $\omega\in\Omega$, and the following holds in $\mathcal{S}'(\mathbb{R}^d)$,
$$u(t\wedge \tau)=u_0+\int_0^{t\wedge\tau}[\Delta u-\mathbf{P}div(u\otimes u)] ds+\int_0^{t\wedge\tau}\mathbf{P}f(t,u)dW_H~~\mathbb{P}-a.s.$$
\end{itemize}
for all $0\leq t<\infty$.

$(2)$ We say that the local strong solution is unique (or indistinguishable) if, given any pair $(u^{(1)},\tau^{(1)}),(u^{(2)},\tau^{(2)})$ of
local strong solutions,
\begin{equation}
\mathbb{P}(1_{u^{(1)}_0=u^{(2)}_0}(u^{(1)}(t)-u^{(2)}(t))=0;\forall t\in [0,\tau^{(1)}\wedge\tau^{(2)}])=1.\label{3.4}
\end{equation}

\end{defn}

Now we state the main results of this paper. The first result concerns the existence and uniqueness of the local strong solution.
\begin{thm}[Local existence of uniqueness solutions]
\label{C2}
Fix a stochastic basis $\mathcal{S}:=(\Omega,\mathscr{F},\mathbb{P},\{\mathscr{F}_t\}_{t\geq0},W_H)$. Assume that  $d\geq 2$, $p$ and $r$ satisfy  (\ref{pr}),
 $u_0\in L^r(\Omega;\dot{B}_{p,r}^{\frac{d}{p}-1})$  is $\mathscr{F}_0$-measurable, f satisfies
\begin{equation}
\label{3.5}
\|\mathbf{P}f(t,u)\|_{\dot{\mathbb{B}}_{p,r}^{\frac{d}{p}-2+\frac{2}{r}}}^r\leq
 \beta_1(t,\|u\|_{\dot{B}_{p,r}^{\frac{d}{p}-1}}
)+\gamma(t)\|u\|_{\dot{{B}}_{p,r}^{\frac{d}{p}-1+\frac{2}{r}}}^r,
\end{equation}
\begin{equation}
\label{3.6}
\|\mathbf{P}f(t,u)-\mathbf{P}f(t,v)\|^r_{\dot{\mathbb{B}}_{p,r}^{\frac{d}{p}-2+\frac{2}{r}}}\leq \beta_2(t,\|(u,v)\|_{\dot{B}_{p,r}^{\frac{d}{p}-1}})\|u-v\|_{\dot{{B}}_{p,r}^{\frac{d}{p}-1}}^r
+\gamma(t)\|u-v\|_{\dot{{B}}_{p,r}^{\frac{d}{p}-1+\frac{2}{r}}}^r,
\end{equation}
 where $\|\beta_i(\cdot,x)\|_{L^\infty}(i=1,2)$ is an increasing and locally bounded function of $x$, $\|\gamma(t)\|_{L^\infty}$ sufficiently small.
Then there exists a unique  local strong solution $(u,\tau)$ of (\ref{1.1}) in the sense of Definition \ref{C1}.
Moreover,
\begin{equation}
\label{3.7}
\mathbb{P}(\tau>0)=1.
\end{equation}
\end{thm}
Assume further that $\beta_1$ has special structure,  $\beta_1(t,x)=\eta(t)x^r$, then we have the global
existence in probability as follows
\begin{thm}
\label{C3}
Under the assumptions of Theorem \ref{C2}, if
$\beta_1(t,x)=\eta(t)x^r$ with $\|\eta(t)\|_{L^\infty}$ sufficiently small.
Then for any $\epsilon>0$, there exists some $\delta>0$ such that if
\begin{equation}
\mathbb{E}(\|u_0\|_{\dot{B}_{p,r}^{\frac{d}{p}-1}}^r)<\delta,
\end{equation}
then $u$ is a unique solution of (\ref{1.1}) on $[0,\tau)$ and $\tau$ satisfies
\begin{equation}
\label{3.9}
\mathbb{P}(\{\tau=\infty\})>1-\epsilon.
\end{equation}
\end{thm}

\begin{rem}
When $d=2$, there are many works studying the global existence and uniqueness of the solution in bounded or unbounded domain, see e.g.
  \cite{brzezniak2013existence,flandoli1995martingale}. Otherwise, for $d\geq3$, it is an open problem even in the deterministic case.
\end{rem}

\begin{rem}
We point out that our conditions of the noise $f(t,u)$ is more general than the linear map of $u$. After suitable modifying, we cover the result of   \cite{sango2016harmonic}, which consider the MHD equations with the additive noise in $\dot{B}_{2,r}^{\frac{d}{2}-1}$. However, when $p>d$, one may take the initial velocity in a Besov space with a \textit{negative} index of regularity, so that a highly oscillating \lq\lq large \rq\rq initial velocity $u_0^\epsilon$ in (\ref{ep}) may give rise to a unique global solution.
\end{rem}
We shall prove Theorems \ref{C2} and \ref{C3} in several steps. In the deterministic setting, the system (\ref{1.1}) can be solved by means of the contraction mapping argument. The so-obtained local in time strong solution exists on a maximal time interval, the length of which can be estimated in terms of norms of the initial data and the external force. However, in the stochastic setting, the lower bounded of the lifespan is  a random variable. Therefore, we  work to the modified system (\ref{4.3}) by introducing two cut-off functions $\theta_1(t)$ and
$\theta_{2,N}(t)$ which  control $\|u\|_{L_t^r\dot{B}_{p,r}^{\frac{d}{p}-1+\frac{2}{r}}}$ and  $\|u\|_{L^\infty_t\dot{B}_{p,r}^{\frac{d}{p}-1}}$,
respectively. According to the bilinear estimate in Lemma \ref{F4}, the heat flow smooth effect in Proposition \ref{F5} , we prove that the mapping (\ref{4.6}), the corresponding mapping of the system (\ref{4.3}), is  a contraction mapping. By the classical Banach fixed point Theorem, we obtain the
  global solutions of (\ref{4.3}).

Then we drop the cut-off function after defining a suitable stopping time, and prove the  existence and uniqueness of the local strong solution for the system (\ref{1.1}) in the sense of Definition \ref{C1}. Finally ,we can estimate the low bound of stopping time  and finish the proof of Theorems
\ref{C2} and \ref{C3} by Chebyshev's inequality and time-space estimates.

In Appendix, we utilize \lq\lq It\^{o} isomorphism\rq\rq (\ref{2.2}),
Da Prato-Kwapi\'{e}n-Zabczyk factorization argument and stochastic Fubini theorem to prove the heat flow smooth effect in stochastic case in Proposition \ref{F5}. It is the particular case of Theorem 3.5 in    \cite{van2012maximal}, but the proof is more simple. For the special case of Proposition \ref{F5} in $\dot{B}_{2,r}^s$, we would refer the reader to   \cite{sango2016harmonic}.

\section{Existence of solution to the stochastic modified equations}
\label{sec:4}
Fix $0<R<1$, which to be determined in (\ref{4.14}), construct two continuous nonincreasing functions $\theta_1:[0,\infty)\mapsto[0,1]$
\begin{equation}
\label{4.1}
\theta_1(x)=\left\{\begin{array}{l}
1,~~~~~~~~~~~~\text{for}~~x<R,\\
2-\frac{x}{R},~~~~~~\text{for}~~R\leq x\leq 2R,\\
0,~~~~~~~~~~~~\text{for}~~x>2R,
\end{array}
\right.
\end{equation}
and $\theta_{2,N}[0,\infty):\mapsto[0,1]$
\begin{equation}
\label{4.2}
\theta_{2,N}(x)=\left\{\begin{array}{l}
1,~~~~~~~~~~~~~~~~~~~\text{for}~~x<N,\\
N+1-x,~~~~~~~\text{for}~~N\leq x\leq N+1,\\
0,~~~~~~~~~~~~~~~~~~~\text{for}~~x>N+1.
\end{array}
\right.
\end{equation}
We consider the following Cauchy problem
\begin{equation}
\label{4.3}
\left\{\begin{array}{l}
du+(\mathbf{P}\text{div}(\chi_{1,u}u\otimes u)-\Delta u)dt=\mathbf{P}\chi_{1,u}\chi_{2,u}f(t,u)dW_H,\\
u|_{t=0}=u_0,
\end{array}
\right.
\end{equation}
where $\chi_{1,u}(t):=\theta_1\big(\|u\|_{L_t^r\dot{B}_{p,r}^{\frac{d}{p}-1+\frac{2}{r}}}\big)$ and $\chi_{2,u}(t):=
\theta_{2,N}\big(\|u\|_{\dot{B}_{p,r}^{\frac{d}{p}-1}}\big)$.

\begin{defn}
A process $u$ is called a global strong solution of $(\ref{4.3})$ if it satisfies the following two conditions:

\begin{itemize}
\item $u$ is right continuous progressively measurable progress and $u\in L^r(\Omega;\mathcal{L}^\infty_{loc}([0,\infty);\dot{B}_{p,r}^{\frac{d}{p}-1}))
\cap L^r_{loc}(\Omega;L^r([0,\infty);\dot{B}_{p,r}^{\frac{d}{p}-1+\frac{2}{r}})).$

\item $u(t,x)\in C([0,\infty);\dot{B}_{p,r}^{\frac{d}{p}-1})$ for almost all $\omega\in\Omega$, and the following identity holds in $\mathcal{S}'(\mathbb{R}^d)$,
$$u(t)=u_0+\int_0^t[\Delta u-\mathbf{P}div(\chi_{1,u}u\otimes u)] ds
+\int_0^t\mathbf{P}\chi_{1,u}\chi_{2,u}f(t,u)dW_H~~\mathbb{P}-a.s.$$
for all $0\leq t<\infty$.
\end{itemize}
\end{defn}

To solve $(\ref{4.3})$ is equivalent to solve the following equation
$$u(t)=e^{t\Delta}u_0+B(\chi_{1,u}u,u)+F_{\chi_{1,u}\chi_{2,u}f}(u),$$
where $B(u,v)(t) $ and $G(u,v)(t)$ are the solution to the heat equation
\begin{equation}
\left\{\begin{split}
\partial_t B(u,v)-\Delta B(u,v)&=\mathbf{P}div(u\otimes v),\\
B(u,v)|_{t=0}&=0,
\end{split}
\right.
\end{equation}
and stochastic heat equation
\begin{equation}
\left\{\begin{aligned}
dF_f(u)-\Delta F_f(u)dt&=\mathbf{P}f(t,u)dW_H,\\
F_{f}(u)|_{t=0}&=0,
\end{aligned}
\right.
\end{equation}
respectively.

\begin{prop}
\label{jieduancunai}
There exists a  global strong solution to the system (\ref{4.3}).
\end{prop}

\begin{proof}
Define
\begin{equation}
\label{4.6}
K(u)=e^{t\Delta}u_0+B(\chi_{1,u}u,u)+F_{\chi_{1,u}\chi_{2,u}f}(u).
\end{equation}
For any $T \geq 0$. By (\ref{3.5}), Lemmas \ref{F2}--\ref{F3}
and
Proposition \ref{F5}, we have
\begin{eqnarray}
\label{4.7}
&&\mathbb{E}\|K(u)\|^r_{\mathcal{L}^\infty_{{T}}\dot{B}_{p,r}^{\frac{d}{p}-1}}+
\mathbb{E}\|K(u)\|^r_{{L}^r_{{T}}\dot{B}_{p,r}^{\frac{d}{p}-1+\frac{2}{r}}}\nonumber\\
&\leq& C\big( \mathbb{E}\|u_0\|^r_{\dot{B}_{p,r}^{\frac{d}{p}-1}}
+\mathbb{E}\|\chi_{1,u}u\|^r_{{L}^r_{{T}}\dot{B}_{p,r}^{\frac{d}{p}-1+\frac{2}{r}}}
\|u\|^r_{{L}^r_{T}\dot{B}_{p,r}^{\frac{d}{p}-1+\frac{2}{r}}}
 +\mathbb{E}\|\mathbf{P}\chi_{1,u}\chi_{2,u}f(t,u)\|^r_{{L}^r_{{T}}\dot{\mathbb{B}}_{p,r}^{\frac{d}{p}-2+\frac{2}{r}}}\big)\nonumber\\
&\leq& C\big(\mathbb{E}\|u_0\|^r_{\dot{B}_{p,r}^{\frac{d}{p}-1}}+R^r\|u\|_{L_T^r\dot{B}_{p,r}^{\frac{d}{p}-1+\frac{2}{r}}}^r
+\|\beta_1(t,N+1)\|_{L^1_T}+
\|\gamma\|_{L^\infty}\|u\|_{L_T^r\dot{B}_{p,r}^{\frac{d}{p}-1+\frac{2}{r}}}^r\big).
\end{eqnarray}
Considering $K(u)-K(v)$, we have
\begin{eqnarray}
K(u)-K(v)&=&\big(B(\chi_{1,u}u,u)-B(\chi_{1,v}v,v)\big)
+\big(F_{\chi_{1,u}\chi_{2,u}f}(u)-F_{\chi_{1,v}\chi_{2,v}f}(v)\big)\nonumber\\
&=&I_1+I_2.
\end{eqnarray}
To estimate $I_1$, we divide into following three cases:

\noindent$(1)$ $\chi_{1,u}(t)>0,\chi_{1,v}(t)>0$,
$$
\chi_{1,u}u\otimes u - \chi_{1,v}v\otimes v 
=(\chi_{1,u}-\chi_{1,v})u\otimes u + \chi_{1,v}(u-v)\otimes u + \chi_{1,v}v\otimes (u-v);
$$
$(2)$ $\chi_{1,u}(t)>0,\chi_{1,v}(t)=0$,
$$
\chi_{1,u}u\otimes u - \chi_{1,v}v\otimes v =(\chi_{1,u}-\chi_{1,v})u\otimes u ;
$$
$(3)$ $\chi_{1,u}=0,\chi_{1,v}>0$,
$$
\chi_{1,u}u\otimes u - \chi_{1,v}v\otimes v = (\chi_{1,u}-\chi_{1,v})v\otimes v .
$$
Then   applying Lemma \ref{F3}, we obtain
\begin{eqnarray}
\label{4.9}
\|I_1\|_{\mathcal{L}_T^\infty\dot{B}_{p,r}^{\frac{d}{p}-1}}^r+\|I_1\|_{L_T^r\dot{B}_{p,r}^{\frac{d}{p}-1+\frac{2}{r}}}^r
&\leq& CR^{2r}\|\chi_{1,u}-\chi_{1,v}\|_{L_T^\infty}^r+CR^r\|u-v\|_{L_T^r\dot{B}_{p,r}^{\frac{d}{p}-1+\frac{2}{r}}}^r\nonumber\\
&\leq& CR^r\|u-v\|_{L_T^r\dot{B}_{p,r}^{\frac{d}{p}-1+\frac{2}{r}}}^r.
\end{eqnarray}
To estimate $I_2$, we also consider three cases:

\noindent$(1)$ $\chi_{1,u}\chi_{2,u}(t)>0,\chi_{1,v}\chi_{2,v}(t)>0,$
\begin{eqnarray*}
&&\textbf{P}\chi_{1,u}\chi_{2,u}f(t,u)-\textbf{P}\chi_{1,v}\chi_{2,v}f(t,v)\\
&=&(\chi_{1,u}-\chi_{1,v})\textbf{P}\chi_{2,u}f(t,u)+(\chi_{2,u}-\chi_{2,v})
\textbf{P}\chi_{1,v}f(t,u)
+\textbf{P}\chi_{1,v}\chi_{2,v}(f(t,u)-f(t,v)),
\end{eqnarray*}
$(2)$ $\chi_{1,u}\chi_{2,u}(t)>0,\chi_{1,v}\chi_{2,v}(t)=0,$
\begin{eqnarray*}
&&\textbf{P}\chi_{1,u}\chi_{2,u}f(t,u)-\textbf{P}\chi_{1,v}\chi_{2,v}f(t,v)\\
&=&(\chi_{1,u}-\chi_{1,v})\textbf{P}\chi_{2,u}f(t,u)+(\chi_{2,u}-\chi_{2,v})
\textbf{P}\chi_{1,v}f(t,u),
\end{eqnarray*}
$(3)$ $\chi_{1,u}\chi_{2,u}(t)=0,\chi_{1,v}\chi_{2,v}(t)>0,$
\begin{eqnarray*}
&&\textbf{P}\chi_{1,u}\chi_{2,u}f(t,u)-\textbf{P}\chi_{1,v}\chi_{2,v}f(t,v)\\
&=&(\chi_{1,u}-\chi_{1,v})\textbf{P}\chi_{2,u}f(t,v)+(\chi_{2,u}-\chi_{2,v})
\textbf{P}\chi_{1,v}f(t,v).
\end{eqnarray*}
Thanks to (\ref{3.5}), (\ref{3.6}) and Proposition \ref{F5}, we obtain
\begin{eqnarray}
\label{4.10}
&&\mathbb{E}\|I_2\|_{\mathcal{L}_T^\infty\dot{B}_{p,r}^{\frac{d}{p}-1}}^r
+\mathbb{E}\|I_2\|_{L_T^r\dot{B}_{p,r}^{\frac{d}{p}-1+\frac{2}{r}}}^r\nonumber\\
&\leq &C\big((||\chi_{1,u}-\chi_{1,v}||_{L_T^\infty}^r
+\|\chi_{2,u}-\chi_{2,v}\|_{L_T^\infty}^r)(\|\beta_1(\cdot,N+1)\|_{L_T^1}
+\|\gamma\|_{L^\infty}R^r)\nonumber\\
&&+\|\beta_2(\cdot,2N+2)\|_{L_T^1}\|u-v\|_{L_T^\infty\dot{B}_{p,r}^{\frac{d}{p}-1}}^r
+\|\gamma\|_{L^\infty}\|u-v\|_{L_T^r\dot{B}_{p,r}^{\frac{d}{p}-1+\frac{2}{r}}}^r\big)\nonumber\\
&\leq& C\big((\frac{1}{R^r}\|u-v\|_{L_T^r\dot{B}_{p,r}^{\frac{d}{p}-1+\frac{2}{r}}}^r+
\|u-v\|_{L_T^\infty\dot{B}_{p,r}^{\frac{d}{p}-1}}^r)(\|\beta_1(\cdot,N+1)\|_{L_T^1}
+\|\gamma\|_{L^\infty}R^r)\nonumber\\
&&+\|\beta_2(\cdot,2N+2)\|_{L_T^1}\|u-v\|_{L_T^\infty\dot{B}_{p,r}^{\frac{d}{p}-1}}^r
+\|\gamma\|_{L^\infty}\|u-v\|_{L_T^r\dot{B}_{p,r}^{\frac{d}{p}-1+\frac{2}{r}}}^r\big)\nonumber\\
&\leq &C\big(\big( T(\|\beta_1(\cdot,N+1)\|_{L^\infty}+\beta_2(\cdot,2N+2)\|_{L^\infty})
+\|\gamma\|_{L^\infty}R^r\big)\|u-v\|_{L_T^\infty\dot{B}_{p,r}^{\frac{d}{p}-1}}^r\nonumber\\
&&+\big(\frac{T\|\beta_1(\cdot,N+1)\|_{L^\infty}}{R^r}+\|\gamma\|_{L^\infty}\big)
\|u-v\|_{L_T^r\dot{B}_{p,r}^{\frac{d}{p}-1+\frac{2}{r}}}^r\big).
\end{eqnarray}
Set
\begin{eqnarray*}
S_T&=\bigg\{&w|w~\text{is progressively measurable,}\\
&&\text{and}~w\in L^r(\Omega;C([0,T];\dot{B}_{p,r}^{\frac{d}{p}-1}))
\cap L^r(\Omega;L^r(0,T;\dot{B}_{p,r}^{\frac{d}{p}-1+\frac{2}{r}}))\bigg\}
\end{eqnarray*}
with the norm $\|w\|_{S_T}^r=\mathbb{E}\|w\|_{\mathcal{L}^\infty_{{T}}\dot{B}_{p,r}^{\frac{d}{p}-1}}^r
+\mathbb{E}\|w\|_{L^r_T\dot{B}_{p,r}^{\frac{d}{p}-1+\frac{2}{r}}}^r$.
 Combining (\ref{4.7})-(\ref{4.10}) together, we obtain
\begin{equation}
\label{4.11}
\|K(u)\|_{S_{{T}}}^r\leq C^*\left(\mathbb{E}\|u_0\|_{\dot{B}_{p,r}^{\frac{d}{p}-1}}^r+T\|\beta_1(\cdot,N+1)\|_{L^\infty}
+(R^r+\|\gamma\|_{L^\infty})\|u\|_{S_{{T}}}^r\right)
\end{equation}
and
\begin{eqnarray}
\label{4.12}
\|K(u)-K(v)\|_{S_T}^r
&\leq& C^*\big((R^r+1)\|\gamma\|_{L^\infty}+(\frac{T}{R^r}+T)
(\|\beta_1(\cdot,N+1)\|_{L^\infty}\nonumber\\
&&+\|\beta_2(\cdot,2N+2)\|_{L^\infty})\big)\|u-v\|_{S_T}^r.
\end{eqnarray}
Define
\begin{equation}
S_{T,M}=\{w|w\in S_T,\|w\|_{S_T}^r\leq M\},
\end{equation}
\begin{equation}
\label{4.14}
R=\min\{1,(4C^*)^{-\frac{1}{r}}\},
\end{equation}
\begin{equation}
M=3C^*\mathbb{E}\|u_0\|_{\dot{B}_{p,r}^{\frac{d}{p}-1}}^r+1,
\end{equation}
and
\begin{equation}
\hat{T}=\frac{1}{3 C^*\big((\frac{1}{R^r}+1)\|\beta_1(\cdot,N+1)\|_{L^\infty}+\|\beta_2(\cdot,2N+2)\|_{L^\infty}\big)}.
\end{equation}
Then if
\begin{equation}
C^*\|\gamma\|_{L^\infty}<\frac{1}{4},\label{gamma}
\end{equation}
then (\ref{4.11}) and (\ref{4.12}) imply the mapping $K$ is a contracting mapping from $S_{\hat{T},M}$ into itself.
By contraction mapping principle, we know that the system (\ref{4.3}) has a unique strong solution $u$   on $[0,\hat{T}]$, where $\hat{T}>0$ is independent of $u_0$. Then we repeat this process to obtain a global strong solution $u$ of the system (\ref{4.3}).
\end{proof}
\section{Proof of  main theorems}
In this section, we shall prove the main theorems  of this paper.
\subsection{Proof of  Theorem $\ref{C2}$.}
By Proposition \ref{jieduancunai}, for each $N\geq 1$, there exists a strong solution $u_N$ for  the system (\ref{4.3})
such that $u_N\in L^r(\Omega;C([0,\infty);\dot{B}_{p,r}^{\frac{d}{p}-1}))
\cap L^r(\Omega;L_{loc}^r(0,\infty);\dot{B}_{p,r}^{\frac{d}{p}-1+\frac{2}{r}})$. We define stopping time
\begin{equation}
\label{5.1}
\begin{split}
\tilde{\sigma}(\omega)&=\left\{\begin{array}{l}
\text{inf}\bigg\{0\leq t<\infty:\|u_N\|_{L_t^r\dot{B}_{p,r}^{\frac{d}{p}-1+\frac{2}{r}}}\geq R\bigg\},\\
\infty,\text{if the above set $\{\cdot\}$ is empty},
\end{array}
\right.\\
\tilde{\varrho}_N(\omega)&=\left\{\begin{array}{l}
\text{inf}\bigg\{0\leq t<\infty:\|u_N\|_{\dot{B}_{p,r}^{\frac{d}{p}-1}}\geq N\bigg\},\\
\infty,\text{if the above set $\{\cdot\}$ is empty},
\end{array}
\right.\\
\tilde{\tau}_N&=\min\{\tilde{\sigma},\tilde{\varrho}_N\}.
\end{split}
\end{equation}
Then, on $[0,\tilde{\tau}]$, where $\tilde{\tau}=\tilde{\tau}_{N_1}\wedge
\tilde{\tau}_{N_2}:=\min\{\tilde{\tau}_{N_1},\tilde{\tau}_{N_2}\}$,   we have
\begin{eqnarray*}
u_{N_1}-u_{N_2}
&=&B(u_{N_1},u_{N_2})-B(u_{N_2,},u_{N_2})+F_f(u_{N_1})-F_f(u_{N_2})\\
&=&B(u_{N_1}-u_{N_2},u_{N_1})+B(u_{N_2},u_{N_1}-u_{N_2})+F_f(u_{N_1})-F_f(u_{N_2}).
\end{eqnarray*}
By (\ref{3.6}), Lemma \ref{F3} and Proposition \ref{F5}, we obtain
\begin{eqnarray}
\label{5.2}
&&\mathbb{E}\left(\|u_{N_1}-u_{N_2}\|_{L_{T\wedge\tilde{\tau}}^\infty\dot{B}_{p,r}^{\frac{d}{p}-1}}^r
+\|u_{N_1}-u_{N_2}\|_{L_{T\wedge\tilde{\tau}}^r\dot{B}_{p,r}^{\frac{d}{p}-1+\frac{2}{r}}}^r\right)\nonumber\\
& \leq& C^*\bigg(2R^r\mathbb{E}\|u_{N_1}-u_{N_2}\|_{L_{T\wedge\tilde{\tau}}^r\dot{B}_{p,r}^{\frac{d}{p}-1+\frac{2}{r}}}^r
+\|\beta_2(\cdot,2N)\|_{L^\infty}\int_0^T\mathbb{E}\|u_{N_1}-u_{N_2}\|_{L_{s\wedge\tilde{\tau}}^\infty
\dot{B}_{p,r}^{\frac{d}{p}-1}}^r ds\nonumber\\
&&+\|\gamma\|_{L^\infty}\mathbb{E}\|u_{N_1}-u_{N_2}
\|_{L_{T\wedge\tilde{\tau}}^r\dot{B}_{p,r}^{\frac{d}{p}-1+\frac{2}{r}}}^r\bigg)\nonumber\\
& \leq& \frac{3}{4}\mathbb{E}\|u_{N_1}-u_{N_2}\|_{L_{T\wedge\tilde{\tau}}^r\dot{B}_{p,r}^{\frac{d}{p}-1+\frac{2}{r}}}^r
+C^*\|\beta_2(\cdot,2N)\|_{L^\infty}\int_0^T\mathbb{E}\|u_{N_1}-u_{N_2}\|_{L_{s\wedge\tilde{\tau}}^\infty
\dot{B}_{p,r}^{\frac{d}{p}-1}}^r ds
\end{eqnarray}
for any $T\geq 0$ and $N=N_1\vee N_2:=\max\{N_1,N_2\},$ where we make use of the definition of $R$ in (\ref{4.14}) and the smallness of $\|\gamma\|_{L^\infty}$ in (\ref{gamma}). Therefore, rearranging the inequality above and invoking the
Gr\"{o}nwall lemma, we conclude that
\begin{equation}
\mathbb{E}\sup\limits_{s\in [0,T\wedge\tilde{\tau}]}\|u_{N_1}-u_{N_2}\|_{\dot{B}_{p,r}^{\frac{d}{p}-1}}^r=0.
\end{equation}
Since $T$ is arbitrary, we obtain
\begin{equation}
\label{5.4}
u_{N_1}=u_{N_2}~~\text{on}~~ [0,\tilde{\tau}].
\end{equation}
Now assume that $N_1<N_2$, hence we have
\begin{equation}
\label{5.5}
\tilde{\tau}_{N_1}(\omega)\leq\tilde{\tau}_{N_2}(\omega)~~\text{for almost all}~\omega\in\Omega.
\end{equation}
If not, (\ref{5.1}) and (\ref{5.4}) yield
$$\|u_{N_1}(\tilde{\tau}_{N_2})\|_{\dot{B}_{p,r}^{\frac{d}{p}-1}}\geq N_2,$$
which is a contradiction with the definition of $\tilde{\tau}_N$ in (\ref{5.1}).
Let
\begin{equation}
\label{e6}
\lim\limits_{N\rightarrow\infty}\tilde{\tau}_N(\omega)=\tau_*~~\text{for almost all}~\omega\in\Omega.
\end{equation}
For $0\leq t\leq  {\tau}_*(\omega)$, we define $u(t)=u_N(t)$, where $\tilde{\tau}_N\geq t$. Then
$u$ is well defined by (\ref{5.4}) and (\ref{5.5}). It holds that
$$u(\omega)\in C([0,\tau_*(\omega));\dot{B}_{p,r}^{\frac{d}{p}-1})\cap L_{loc}^r(0,\tau_*(\omega);\dot{B}_{p,r}^{\frac{d}{p}-1+\frac{2}{r}})~~\mathbb{P}-a.s..$$
Then, $\tau_N$ defined by (\ref{3.2}) in term of this $u$ is the same as $\tilde{\tau}_N$ defined by (\ref{5.1}),
and $\tau$ defined by (\ref{3.3}) is the same as $\tau_*$. If $\tau(\omega)<\infty$, we set $u(t)=0$ for $t\geq \tau(\omega)$.
Therefore, $(u,\tau)$ is obviously a solution in the sense of Definition \ref{C1}. If there are two solutions $(u^1,\tau^1)$
and $
(u^2,\tau^2)$, similarly to the above argument in (\ref{5.2})-(\ref{5.4}), we can easily show $u^1=u^2$ on $[0,\tau_N^1\wedge\tau_N^2]$ for each $N\geq 1$. Hence we obtain the uniqueness of the local strong solution in the sense of (\ref{3.4}).

Next, we shall prove (\ref{3.7}). As similar as (\ref{4.7}), for all $\delta>0$, it holds
\begin{eqnarray}
\mathbb{E}\int_0^{\delta\wedge\tau_N}\|u(s)\|^r_{\dot{B}_{p,r}^{\frac{d}{p}-1+\frac{2}{r}}}ds
&\leq& C^*\bigg(\mathbb{E}\|e^{t\Delta}u_0\|_{L_T^r\dot{B}_{p,r}^{\frac{d}{p}-1+\frac{2}{r}}}^r
+\delta\|\beta_1(t,N)\|_{L^\infty}\nonumber\\
&&+(R^{r}+\|\gamma(t)\|_{L^\infty})\mathbb{E}\int_0^{\delta\wedge\tau_N}\|u(s)\|^r_{\dot{B}_{p,r}^{\frac{d}{p}-1+\frac{2}{r}}}ds
\bigg),
\end{eqnarray}
and
\begin{eqnarray}
\mathbb{E}\sup\limits_{0\leq s\leq \delta}\|u(s\wedge\tau_N)\|_{\dot{B}_{p,r}^{\frac{d}{p}-1}}^r
&\leq& C^*\bigg(\mathbb{E}\|u_0\|_{\dot{B}_{p,r}^{\frac{d}{p}-1}}^r
+\delta\|\beta_1(t,N)\|_{L^\infty}\nonumber\\
&&+(R^{r}+\|\gamma(t)\|_{L^\infty})\mathbb{E}\int_0^{\delta\wedge\tau_N}\|u(s)\|^r_{\dot{B}_{p,r}^{\frac{d}{p}-1+\frac{2}{r}}}ds
\bigg).
\end{eqnarray}
According to the definition of $R$ in (\ref{4.14}) and the smallness of $\|\gamma\|_{L^\infty}$ in (\ref{gamma}), we have
\begin{equation}
\label{5.9}
\mathbb{E}\int_0^{\delta\wedge\tau_N}\|u(s)\|^r_{\dot{B}_{p,r}^{\frac{d}{p}-1+\frac{2}{r}}}ds
\leq 2C^*\big(\mathbb{E}\|e^{t\Delta}u_0\|_{L_\delta^r\dot{B}_{p,r}^{\frac{d}{p}-1+\frac{2}{r}}}^r
+\delta\|\beta_1(t,N)\|_{L^\infty}\big),
\end{equation}
\begin{equation}
\label{5.10}
\mathbb{E}\sup\limits_{0\leq s\leq \delta}\|u(s\wedge\tau_N)\|_{\dot{B}_{p,r}^{\frac{d}{p}-1}}^r
\leq 2C^*\big(\mathbb{E}\|u_0\|_{\dot{B}_{p,r}^{\frac{d}{p}-1}}^r
+\delta\|\beta_1(t,N)\|_{L^\infty}\big).
\end{equation}
By the definition of $\tau_N$, one has
\begin{equation*}
\{\omega|\tau_N(\omega)\leq \delta\}\subset\bigg\{\omega\big|\sup\limits_{0\leq s\leq \delta}
\|u(s\wedge\tau_N)\|
_{\dot{B}_{p,r}^{\frac{d}{p}-1}}\geq N\bigg\}\cup\bigg\{\omega\big|\int_0^{\delta\wedge\tau_N}\| u(s)\|_{\dot{B}_{p,r}^{\frac{d}{p}-1+\frac{2}{r}}}^r ds\}\geq R^r\bigg\}.
\end{equation*}
By Chebyshev's inequality,  for any $N\in \mathbb{N}^+$, it follows from (\ref{5.9}) and (\ref{5.10}) that
\begin{eqnarray}
\label{5.11}
\mathbb{P}(\{\tau_N\leq \frac{1}{k}\})
&\leq& \frac{2C^*}{N^r}\bigg(\mathbb{E}\|u_0\|_{\dot{B}_{p,r}^{\frac{d}{p}-1}}^r
+\frac{1}{k}\|\beta_1(t,N)\|_{L^\infty}\bigg)\nonumber\\
&&+\frac{2C^*}{R^r}\bigg(\mathbb{E}\|e^{t\Delta}u_0\|_{L_{\frac{1}{k}}^r\dot{B}_{p,r}^{\frac{d}{p}-1+\frac{2}{r}}}^r
+\frac{1}{k}\|\beta_1(t,N)\|_{L^\infty}\bigg).
\end{eqnarray}
As $u_0\in L^r(\Omega;\dot{B}_{p,r}^{\frac{d}{p}-1})$, for any $\varepsilon<1$, a positive $J_{\varepsilon,u_0}$ exists
such that
\begin{equation*}
\sum\limits_{j\geq J_{\varepsilon,u_0}}\mathbb{E} 2^{jr(\frac{d}{p}-1)}\|\Delta_j u_0\|_{L^p}^r\leq\frac{ R^r\varepsilon}{16C^*}.\\
\end{equation*}
Thus, we have
\begin{eqnarray}
\label{5.12}
\mathbb{E}\|e^{t\Delta}u_0\|_{L_{\frac{1}{k}}^r\dot{B}_{p,r}^{\frac{d}{p}-1+\frac{2}{r}}}^r
&=&\sum\limits_{j=-\infty}^\infty \mathbb{E} 2^{jr(\frac{d}{p}-1+\frac{2}{r})}\|e^{t\Delta }\Delta_j u_0\|^r_{L_{\frac{1}{k}}^rL^p}\nonumber\\
&\leq& \frac{R^r\varepsilon}{16C^*}+\sum\limits_{j\leq J_{\varepsilon,U_0}}\mathbb{E} 2^{jr(\frac{d}{p}-1+\frac{2}{r})}\|e^{t\Delta }\Delta_ju_0\|^r_{L_{\frac{1}{k}}^rL^p}\nonumber\\
&\leq& \frac{R^r\varepsilon}{16C^*}+\frac{1}{k}2^{2J_{\epsilon,u_0}}\mathbb{E}\|u_0\|_{\dot{B}_{p,r}^{\frac{d}{p}-1}}^r.
\end{eqnarray}
 Using (\ref{5.11}) and (\ref{5.12}), for any $\varepsilon>0$, let $N=\left({16C^*\mathbb{E}\|u_0\|_{\dot{B}_{p,r}^{\frac{d}{p}-1}}^r}{\varepsilon}^{-1}\right)^{\frac{1}{r}}$, and $$k\geq K=\max\{\frac{\|\beta_1(t,N)\|_{L^\infty}}{\mathbb{E}\|u_0\|_{\dot{B}_{p,r}^{\frac{d}{p}-1}}^r},
16C^*\|\beta_1(t,N)\|_{L^\infty}\varepsilon^{-1} R^{-r},C^*2^{2J_{\varepsilon,u_0}+4}\mathbb{E}\|u_0\|_{\dot{B}_{p,r}^{\frac{d}{p}-1}}^r \varepsilon^{-1} R^{-r}\}, $$we have
\begin{equation}
\mathbb{P}(\{\tau_N\leq \frac{1}{k}\})\leq \varepsilon.
\end{equation}
Hence, we get
\begin{eqnarray*}
\mathbb{P}(\{\tau>0\})
&=&1-\mathbb{P}(\{\tau=0\})\\
&\geq& 1-\mathbb{P}(\{\tau_N=0\})\\
&\geq& 1-\lim\limits_{k\rightarrow\infty}\mathbb{P}(\{\tau_N\leq\frac{1}{k}\})\\
&\geq& 1-\varepsilon,
\end{eqnarray*}
then $\mathbb{P}(\{\tau>0\})=1$ since $\varepsilon$ is arbitrary. Therefore, we finish the proof of Theorem
\ref{C2}. \hfill$\Box$

Next, if the initial data are small and $\beta_1(t,x)=\eta(t)x^r$, we shall show the global existence of the strong solution in probability.
\subsection{Proof of Theorem $\ref{C3}$}
\begin{lem}
\label{E1}
Let $(u,\tau)$ be a local strong solution of  the system (\ref{1.1}) in Theorem \ref{C2} and $\sigma$  be the stopping time defined in (\ref{3.2}).
If $\beta_1(t,x)=\eta(t)x^r$ with $\|\eta\|_{L^1}<\frac{1}{C^*}$, then we have
$\sigma\leq\tau$ almost surely on the set $\{\tau<\infty\}$.
\end{lem}
\begin{proof}
As similar as  (\ref{4.7}),  for all $t\geq 0$ we have
\begin{eqnarray}
&&\mathbb{E}\sup\limits_{0\leq s\leq t}\|u(s\wedge\sigma\wedge\tau)\|_{\dot{B}_{p,r}^{\frac{d}{p}-1}}^r
+\mathbb{E}\int_0^{t\wedge\delta\wedge\tau}\|u(s)\|^r_{\dot{B}_{p,r}^{\frac{d}{p}-1+\frac{2}{r}}}ds\nonumber\\
\label{e15}
 &\leq& C^*\bigg(\mathbb{E}\|u_0\|_{\dot{B}_{p,r}^{\frac{d}{p}-1}}^r+(R^r+\|\gamma\|_{L^\infty})
 \mathbb{E}\int_0^{t\wedge\delta\wedge\tau}\|u(s)\|^r_{\dot{B}_{p,r}^{\frac{d}{p}-1+\frac{2}{r}}}ds
 \nonumber\\
 &&+\|\eta\|_{L^1}
 \mathbb{E}\sup\limits_{0\leq s\leq t}\|u(s\wedge\sigma\wedge\tau)\|_{\dot{B}_{p,r}^{\frac{d}{p}-1}}^r\bigg).
\end{eqnarray}
Invoking the definition of $R$ in (\ref{4.14}) and the smallness of $\|\gamma\|_{L^\infty}$ in (\ref{gamma}), we obtain
\begin{equation*}
(1-C^*\|\eta\|_{L^1})\sup\limits_{0\leq s\leq t}\mathbb{E}\|u(s\wedge\sigma\wedge\tau)\|_{\dot{B}_{p,r}^{\frac{d}{p}-1}}^r
 \leq C^*\mathbb{E}\|u_0\|_{\dot{B}_{p,r}^{\frac{d}{p}-1}}^r,
\end{equation*}

then
\begin{equation*}
\sup\limits_{0\leq s\leq t}\mathbb{E}\|u(s\wedge\sigma\wedge\tau)\|_{\dot{B}_{p,r}^{\frac{d}{p}-1}}^r
 \leq \frac{C^*}{1-C^*\|\eta\|_{L^1}}\mathbb{E}\|u_0\|_{\dot{B}_{p,r}^{\frac{d}{p}-1}}^r.
\end{equation*}
By Chebyshev's inequality , we have
\begin{equation*}
\mathbb{P}(\{\tau_N\leq \sigma\})\leq \frac{C^* \mathbb{E}\|u_0\|_{\dot{B}_{p,r}^{\frac{d}{p}-1}}^r}{N^r(1-C^*\|\eta\|_{L^1})}.
\end{equation*}
Therefore,
\begin{equation*}
\mathbb{P}(\{\tau\leq\sigma\})=\lim_{N\rightarrow \infty}\mathbb{P}(\{\tau_N\leq \sigma\})=0.
\end{equation*}
The proof of Lemma \ref{E1} is thus completed.
\end{proof}
By (\ref{e15}), we have
\begin{equation}
\mathbb{E}\int_0^{t\wedge \sigma}\|u\|_{\dot{B}_{p,r}^{\frac{d}{p}-1+\frac{2}{r}}}^r\leq 2C^*\mathbb{E}\|u_0\|_{\dot{B}_{p,r}^{\frac{d}{p}-1}}^r.
\end{equation}
Let us denote
$$\mathcal{G}=\{\sigma<\infty\}.$$
By Fatou's lemma, for $0\leq t<\infty$, we have
\begin{eqnarray*}
R^r\mathbb{P}(\mathcal{G})
&=&\mathbb{E}(\chi_{\mathcal{G}}\lim\limits_{t\rightarrow{\infty}}\int_0^{t\wedge \sigma}\|u\|_{\dot{B}_{p,r}^{\frac{d}{p}-1+\frac{2}{r}}}^r)\\
&\leq&\lim\inf\limits_{t\rightarrow\infty}\mathbb{E}(\chi_{\mathcal{G}}\int_0^{t\wedge \sigma}\|u\|_{\dot{B}_{p,r}^{\frac{d}{p}-1+\frac{2}{r}}}^r)\\
&\leq&2C^*\mathbb{E}\|u_0\|^r_{\dot{B}_{p,r}^{\frac{d}{p}-1}}.
\end{eqnarray*}
Here $\chi_\mathcal{G}$ is the characteristic function of the set $\mathcal{G}$. Hence
$$\mathbb{P}(\mathcal{G})\leq \frac{2C^*}{R^r}\mathbb{E}\|u_0\|_{\dot{B}_{p,r}^{\frac{d}{p}-1}}^r.$$
Therefor, by Lemma \ref{E1},
$$\mathbb{P}(\{\tau=\infty\})
=\mathbb{P}(\{\sigma=\infty\})\geq1-\frac{2C^*}{R^r}\mathbb{E}\|u_0\|_{\dot{B}_{p,r}^{\sigma_p}}^r.$$
Then for any $\epsilon>0$,   choosing $\delta=\frac{\epsilon R^r}{2C^*}$,  we complete the proof of Theorem \ref{C3}.\hfill $\Box$
\section{Appendix}
The following two lemmas describe the bilinear estimate of $\mathbf{P}\text{div}(a\otimes b)$ and
the smothing effect of the heat flow. Then will de found in   \cite{zhang2008wellposedness}, see also
  \cite{Chemin1999Th}.

\begin{lem}[Lemma 4.1 in   \cite{zhang2008wellposedness}]
\label{F1}
Let $(p,q,r)\in [1,2)\times [2,\infty]\times [1,\infty]$ such that $\frac{d}{p}-1+\frac{4}{q}>1$ or $(p,q,r)\in [2,\infty]^2
\times [1,\infty]$ such that $\frac{d}{p}-1+\frac{2}{q}>0$. A constant $C$ exists such that
\begin{equation}
\|\mathbf{P}\text{div}(u\otimes v)\|_{\mathcal{L}^{\frac{q}{2}}([0,T];\dot{B}_{p,r}^{ \frac{d}{p}-3+\frac{4}{q}})}
\leq C \|u\|_{\mathcal{L}^q([0,T];\dot{B}_{p,r}^{\frac{d}{p}-1+\frac{2}{q}})}
\|v\|_{\mathcal{L}^q([0,T];\dot{B}_{p,r}^{\frac{d}{p}-1+\frac{2}{q}})}.
\end{equation}
\end{lem}
\begin{lem}[Lemma 4.2 in   \cite{zhang2008wellposedness}, Proposition 2.1 in   \cite{Chemin1999Th}]
\label{F2}
Let $0<T\leq\infty$, $I=[0,T)$, $1\leq p,r\leq \infty$, $1\leq q\leq q_1\leq\infty$, $s\in\mathbb{R}$.  Denote $u$ as the solution  in
$\mathcal{S}'(\mathbb{R}^d)$ of
\begin{equation}
\left\{\begin{aligned}
\partial_t u-\Delta u&=f,\\
u|_{t=0}&=u_0,
\end{aligned}
\right.
\end{equation}
with $f\in\mathcal{L}^{q}(I;\dot{B}_{p,r}^{s-2+\frac{2}{q}})$ and $u_0\in\dot{B}_{p,r}^{s}$. Then
\begin{equation}
u\in \mathcal{L}^{q_1}(I;\dot{B}_{p,r}^{s+\frac{2}{q}_1})\cap \mathcal{L}^\infty(I;\dot{B}_{p,r}^{s})
\end{equation}
satisfying
\begin{equation}
\|u\|_{\mathcal{L}^{q_1}(I;\dot{B}_{p,r}^{s+\frac{2}{q}_1})}\leq C\left(\|u_0\|_{\dot{B}_{p,r}^s}
+\|f\|_{\mathcal{L}^{q}(I;\dot{B}_{p,r}^{s-2+\frac{2}{q}})}\right).
\end{equation}
Especially, $u\in C(I;\dot{B}_{p,r}^{s})$ when $r<\infty$.
\end{lem}
Combine with Lemmas \ref{F1} and \ref{F2}, we have
\begin{lem}
\label{F3}
If $(p,q,r)\in [2,\infty]^2\times [1,\infty]$
such that $\frac{d}{p}-1+\frac{2}{q}>0$. A constant $C$ exists such that
$$\|B(a,b)\|_{\mathcal{L}^{q}([0,T];\dot{B}_{p,r}^{\frac{d}{p}-1+\frac{2}{q}})}
\leq C\|a\|_{\mathcal{L}^q([0,T];\dot{B}_{p,r}^{\frac{d}{p}-1+\frac{2}{q}})}\|b\|_{\mathcal{L}^q([0,T];\dot{B}_{p,r}^{\frac{d}{p}-1+\frac{2}{q}})}$$
where $B(a,b)$ satisfies
\begin{equation*}
\left\{\begin{array}{l}
\partial_t B(a,b)-\Delta B(a,b)=\mathbf{P}\text{div}(a\otimes b),\\
B(a,b)|_{t=0}=0.
\end{array}
\right.
\end{equation*}
\end{lem}
Now we extend Lemma $\ref{F2}$ to the stochastic case, and  we need
the following lemma which describes the action of the semigroup of the heat equation on distributions with Fourier transforms supported in an annulus.
\begin{lem}[Lemma 2.4 in   \cite{bahouri2011fourier}, Lemma 2.1 in   \cite{Chemin1999Th}]
\label{F4}
Let $\mathcal{C}$ be an annulus. Positive constants $c$ and $C$ exist such that for any $p\in [1,\infty]$ and any couple $(t,\lambda)$ of positive real
numbers, we have
$$\text{supp}~\mathcal{F}u\subset\lambda\mathcal{C}\Rightarrow\|e^{t\Delta}u\|_{L^p}\leq Ce^{-ct\lambda^2}\|u\|_{L^p}.$$
\end{lem}

\begin{prop}
\label{F5}
Let $0<T\leq\infty$, $I=[0,T)$, $2\leq p,r<\infty$, $s\in\mathbb{R}$. Denote $u$ as the solution in $\mathcal{S}'(\mathbb{R}^d)$ of
\begin{equation}
\label{6.5}
\left\{\begin{array}{l}
du-\Delta u dt=GdW_H,\\
u|_{t=0}=0,
\end{array}
\right.~~~~~\mathbb{P}-a.s.
\end{equation}
with $\mathscr{F}$-adapted process $G$ in $L^r(\Omega;L^r(I;\dot{\mathbb{B}}_{p,r}^{s-1}))$. Then for any $r_1\in [r,\infty]$, $u\in L^r(\Omega;\mathcal{L}^{r_1}(I;\dot{B}_{p,r}^{s+\frac{2}{r_1}-\frac{2}{r}})\cap
{C}(I;\dot{B}_{p,r}^{s-\frac{2}{r}}))$ satisfying
\begin{equation}
\label{6.6}
\mathbb{E}\|u\|^r_{\mathcal{L}^{r_1}(I;\dot{B}_{p,r}^{s+\frac{2}{r_1}-\frac{2}{r}})}\leq C\mathbb{E}\|G\|^r_{L^r(I;\dot{\mathbb{B}}_{p,r}^{s-1})}.
\end{equation}
\end{prop}
\begin{proof}
Thanks to $G\in \mathcal{S}_h'(\mathbb{R}^d,H)$ a.s., then the unique solution of (\ref{6.5}) $u\in C(I;\mathcal{S}'(\mathbb{R}^d))$ a.s., satisfying
$$\hat{u}(t,\xi)=\int_0^t e^{(t-\tau)|\xi|^2}\hat{G}(\tau,\xi)dW_H(\tau)~~~~\mathbb{P}-a.s..$$
As $j$ tends to $-\infty$, we have $\mathcal{F}(S_jG)\overset{\mathbb{P}-a.s.}{\rightarrow}0$, implies
$$\mathcal{F}(S_ju)\overset{\mathbb{P}-a.s.}{\rightarrow}0~~\text{in}~C(I;\mathcal{S}'(\mathbb{R}^d))\Rightarrow u\in C(I;\mathcal{S}'_h(\mathbb{R}^d))~~~~\mathbb{P}-a.s..$$
On the other hand,  we get
$$\Delta_ju(t)=\int_0^te^{(t-\tau)\Delta}\Delta_j G(\tau)dW_H(\tau).$$
Then by (\ref{2.2}), Lemma \ref{F4} and Young's inequality, we have
\begin{eqnarray}
\mathbb{E}\|u\|_{L_T^r\dot{B}_{p,r}^s}^r
&\leq& C\sum\limits_{j} 2^{jrs}\int_0^T\mathbb{E}\|\Delta_ju\|_{L^p}^rds\nonumber\\
&\leq& C\sum\limits_{j} 2^{jrs}\int_0^T\mathbb{E}\bigg\|\int_0^t e^{(t-\tau)\Delta}\Delta_j GdW_H\bigg\|_{L^p}^rds\nonumber\\
&\leq& C\sum\limits_{j} 2^{jrs}\int_0^T\mathbb{E}\bigg(\int_0^t\|e^{(t-\tau)\Delta}\Delta_j G\|_{L^p(\mathbb{R}^d;H)}^2\bigg)^{\frac{r}{2}}\nonumber\\
&\leq& C\mathbb{E}\sum\limits_{j} 2^{jrs}\int_0^T \bigg(\int_0^t c_1e^{-2c_2(t-\tau)2^{2j}}\|\Delta_jG\|_{L^p(\mathbb{R}^d;H)}^2\bigg)^{\frac{r}{2}}\nonumber\\
&\leq& C\mathbb{E}\sum\limits_{j} 2^{jr(s-1)}\|\Delta_jG\|_{L_T^rL^p(\mathbb{R}^d;H)}^r\nonumber\\
\label{f7}
&\leq& C\mathbb{E}\|G\|_{L_T^r\dot{\mathbb{B}}_{p,r}^{s-1}}^r.
\end{eqnarray}
For estimating the supper norm of time, we  adopt the Da Prato-Kwapi\'{e}n-Zabczyk factorization argument (see   \cite{brzezniak1997stochastic},  or Section 5.3 of   \cite{da2014stochastic}, and references therein), using stochastic Fubini theorem and the
equality
$$\int_\tau^t (t-s)^{\alpha-1}(s-\tau)^{-\alpha}ds=\Gamma(\alpha)\Gamma(1-\alpha)\ 0<\alpha<1.$$
We obtain
$$\int_0^t e^{(t-\tau)\Delta}\Delta_jG dW_H=\frac{1}{\Gamma(\alpha)\Gamma(1-\alpha)}\int_0^t e^{(t-s)\Delta}(t-s)^{\alpha-1}\int_0^s e^{(s-\tau)\Delta}(s-\tau)^{-\alpha}\Delta_jGdW_H ds$$
for $\alpha\in (0,\frac{1}{2}).$
Then  as similar as (\ref{f7}), we have
\begin{eqnarray}
&&\mathbb{E}\sup\limits_{t\in [0,T]}\bigg\|\int_0^t e^{(t-s)\Delta}\Delta_j G dW_H\bigg\|_{L^p}^r\nonumber\\
&\leq& C\frac{1}{\Gamma(\alpha)\Gamma(1-\alpha)}\mathbb{E}\sup\limits_{t\in [0,T]}\bigg(\int_0^t\bigg\|e^{(t-s)\Delta}(t-s)^{\alpha-1}\int_0^s e^{(s-\tau)\Delta}(s-\tau)^{-\alpha}\Delta_jGdW_H\bigg\|_{L^p} ds\bigg)^r\nonumber\\
&\leq& C\mathbb{E}\sup\limits_{t\in [0,T]}\bigg(\int_0^t c_1e^{-c_22^{2j}(t-s)}(t-s)^{\alpha-1}\bigg\|\int_0^s e^{(s-\tau)\Delta}(s-\tau)^{-\alpha}\Delta_jGdW_H\big\|_{L^p} ds\bigg)^r\nonumber\\
&\leq& C2^{-2jr(\alpha-1+\frac{1}{r'})}\int_0^T\mathbb{E}\bigg\|\int_0^s e^{(s-\tau)\Delta}(s-\tau)^{-\alpha}\Delta_jGdW_H\bigg\|_{L^p}^rds\nonumber\\
&\leq& C 2^{-2jr(\alpha-1+\frac{1}{r'})}\int_0^T\mathbb{E}\bigg(\int_0^s \|e^{ (s-\tau)\Delta}(s-\tau)^{- \alpha}\Delta_j
G\|_{L^p(\mathbb{R}^d;H)}^2 d\tau\bigg)^{\frac{r}{2}}ds\nonumber\\
&\leq& C 2^{-2jr(\alpha-1+\frac{1}{r'})}2^{-2jr(-\alpha+\frac{1}{2})}\mathbb{E}\|\Delta_j G\|_{L_T^rL^p(\mathbb{R}^d;H)}^r\nonumber\\
&\leq& C 2^{j(2-r)}\mathbb{E}\|\Delta_j G\|_{L_T^rL^p(\mathbb{R}^d;H)}^r.
\end{eqnarray}
Then,
\begin{eqnarray}
\mathbb{E} \|u\|_{\mathcal{L}_T^\infty \dot{B}_{p,r}^{s-\frac{2}{r}}}^r
&=&\sum\limits_{j}2^{jr(s-\frac{2}{r})}\mathbb{E}\sup\limits_{t\in [0,T]}\|\Delta_j u\|_{L^p}^r\nonumber\\
&=&\sum\limits_{j}2^{jr(s-\frac{2}{r})}\mathbb{E}\big\|\int_0^t e^{(t-s)\Delta}\Delta_j G dW_H\big\|_{L^p}^r\nonumber\\
&\leq&C\mathbb{E}\sum\limits_{j}2^{jr(s-\frac{2}{r})}2^{j(2-r)}\|\Delta_j G\|_{L_T^rL^p(\mathbb{R}^d;H)}^r\nonumber\\
&=&C\mathbb{E}\sum\limits_{j}2^{jr(s-1)}\|\Delta_j G\|_{L_T^rL^p(\mathbb{R}^d;H)}^r\nonumber\\
&=&C\mathbb{E}\|G\|_{L_T^r\dot{\mathbb{B}}_{p,r}^{s-1}}^r.
\end{eqnarray}
Then (\ref{6.6}) is obtained by the interpolation inequality. This implies the lemma except $u\in L^r(\Omega;{C}(I;\dot{B}_{p,r}^{s-\frac{2}{r}}))$, which is held by $\mathcal{S}\cap \dot{B}_{p,r}^{s-\frac{2}{r}}$ dense in $\dot{B}_{p,r}^{s-\frac{2}{r}}$.
\end{proof}

\section*{acknowledgments}
This work is  supported by Zhejiang Provincial Natural Science Foundation of China LR17A010001, and National Natural Science Foundation of China 11331005  and 11621101.

\end{document}